\documentclass[fontsize=10pt]{scrreprt}

\usepackage[square,numbers,nonamebreak]{natbib}
\usepackage{setspace}

\usepackage[english]{babel}

\usepackage{chngcntr}

\usepackage[hang]{footmisc}
 at 11truept
\font\ssc=pplrc9d at 11 truept
\newcommand\qedbox{$\rlap{$\sqcap$}\sqcup$}

\usepackage[hmarginratio=1:1, bottom=2cm, top=2.5cm]{geometry}

\usepackage[automark]{scrlayer-scrpage}
\cohead{\textnormal{The Schur--Zassenhaus Theorem for finite skew braces}}
\lehead{\pagemark}
\rohead{\pagemark}

\let\ceheadL\cehead
\renewcommand\cehead[1]{
\ceheadL{\textnormal{#1}}
}

\cehead{M. Ferrara --- M. Trombetti}
\lefoot{}
\rofoot{}

\usepackage{graphicx}
\usepackage[usenames,dvipsnames,svgnames,table]{xcolor}
\definecolor{Maroon}{cmyk}{0, 0.87, 0.68, 0.32}
\definecolor{RoyalBlue2}{cmyk}{80,100,0,0.1}

\newcommand\auths[1]{\large \textsc{\textcolor{Maroon}{#1}}\setstretch{1.2}}
\newcommand\titl[1]{\center \linespread{1.1}\color{RoyalBlue2}\Large\textbf{ #1}\color{black}\bigskip} 
\renewcommand\abstract[1]{
\begin{center}
{\textbf{Abstract}}
\end{center}
{
\linespread{1.1}\fontsize{9pt}{-10pt}\selectfont #1}}

\usepackage[utf8]{inputenc}
\usepackage[sc,osf]{mathpazo}
\usepackage[euler-digits]{eulervm}
\usepackage[T1]{fontenc}
\usepackage{mathtools}

\DeclareSymbolFont{operators}{\encodingdefault}{ppl}{m}{n}
\DeclareMathAlphabet{\mathbf}{\encodingdefault}{ppl}{bx}{n}
\DeclareMathAlphabet{\mathit}{\encodingdefault}{ppl}{m}{it}

\usepackage{amsfonts}
\usepackage{amsmath}
\usepackage{ragged2e}
\usepackage{titlesec}
\usepackage{amssymb}

\renewcommand{\thesection}{\arabic{section}}
 
\titleformat{\section}{\medskip\bigskip\normalfont\Large\bf}{\thesection}{0.5em}{}
\titleformat{\subsection}{\smallskip\bigskip\normalfont\large\bf}{\thesubsection}{0.5em}{}

\usepackage{amsthm}
\newtheoremstyle{dotless}{}{}{\itshape}{}{\bfseries}{}{1em}{}

\theoremstyle{dotless}
\newtheorem{theo}{Theorem}

\newtheorem{lem}[theo]{Lemma}

\newtheorem{rem}[theo]{Remark}
\newtheorem{ex}[theo]{Example}
\renewenvironment{proof}{\smallbreak\noindent {\sc Proof \;---\;}}{\hfill\qedbox}

\counterwithout{theo}{section}
\numberwithin{theo}{section}

\DeclareOldFontCommand{\rm}{\normalfont\rmfamily}{\mathrm}
\DeclareOldFontCommand{\sf}{\normalfont\sffamily}{\mathsf}
\DeclareOldFontCommand{\tt}{\normalfont\ttfamily}{\mathtt}
\DeclareOldFontCommand{\bf}{\normalfont\bfseries}{\mathbf}
\DeclareOldFontCommand{\it}{\normalfont\itshape}{\mathit}
\DeclareOldFontCommand{\sl}{\normalfont\slshape}{\@nomath\sl}
\DeclareOldFontCommand{\sc}{\normalfont\scshape}{\@nomath\sc}

\begin{document}

\setheadsepline{1pt}[\color{black}]

\titl{The Schur--Zassenhaus Theorem and Sylow's Third Theorem for Finite Skew Braces
\footnote{The authors are supported by GNSAGA (INdAM) and are members of the non-profit association ``Advances in Group Theory and Applications'' (www.advgrouptheory.com).}}

\auths{M. Ferrara -- M. Trombetti}

\thispagestyle{empty}
\justify\noindent
\setstretch{0.3}
\abstract{In this short note we establish the Schur--Zassenhaus Theorem
and Sylow's Third Theorem for finite skew braces. More precisely, we
prove that every Hall ideal of a finite skew brace admits a sub-skew
brace complement, and more generally that every left ideal whose order is
coprime to that of the Hall ideal can be embedded in such a complement.
Using similar ideas we show that every left ideal of prime-power order is contained in
a Sylow sub-skew brace. Finally, we prove that the number of Sylow
$p$-sub-skew braces is congruent to $1$ modulo $p$, and provide examples
showing that the corresponding containment property fails for arbitrary
sub-skew braces.}

\setstretch{2.1}
\noindent
{\fontsize{10pt}{-10pt}\selectfont {\it Mathematics Subject Classification \textnormal(2020\textnormal)}: Primary 20N99; Secondary 16T25, 20D20}\\[-0.8cm]

\noindent 
\fontsize{10pt}{-10pt}\selectfont  {\it Keywords}: skew brace; Schur--Zassenhaus Theorem; Sylow Theorems\\[-0.8cm]

\setstretch{1.1}
\fontsize{11pt}{12pt}\selectfont
\section{Introduction}

In its classical form, the Schur--Zassenhaus Theorem asserts that if
\(G\) is a finite group and \(N\trianglelefteq G\) is such that
\((|N|,|G/N|)=1\), then \(G\) has a subgroup of order \(|G/N|\), and
any two such subgroups are conjugate in \(G\). Its importance in finite
group theory is well established; as Robinson writes in \cite{Rob95},
it ``must be reckoned as one of the truly fundamental results of group
theory''.

Some partial extensions of the Schur--Zassenhaus Theorem to the setting
of skew braces have been obtained in \cite{RatheeYadav26}. In this note,
we prove the existence part of the Schur--Zassenhaus Theorem for arbitrary
finite skew braces (see Theorem \ref{schurzassenhaus}). Our proof is based on an application of Hall's theorem to a suitable
semidirect product naturally associated with the lambda action of \(B\).
In fact, the same argument proves more: if \(I\) is a Hall ideal of \(B\),
then every left ideal whose order is coprime to \(|I|\) is contained in a
complement of \(I\).

Using similar ideas, we also obtain a containment result in the Sylow
setting: every left ideal of prime-power order is contained in a Sylow
sub-skew brace (see Theorems~\ref{sylowtheorem} and \ref{sylowtheorem2}). A couple of examples show that limit of the theory here (see~Examples~\ref{ex1} and \ref{ex2}). Nevertheless, our second main result shows that there is an analogue of the Sylow's Third Theorem for finite skew braces (see Theorem \ref{thirdsylow}).

\section{Proof of the main result}

We use the following standard notation. If \(B=(B,+,\circ)\) is a skew
brace, then \(\lambda_a(b)=-a+a\circ b\). A {\it left ideal} is an additive
subgroup invariant under all maps \(\lambda_a\). An {\it ideal} is a left ideal
which is normal in both \((B,+)\) and \((B,\circ)\). A {\it sub-skew brace} is
a subset which is a subgroup with respect to both operations.

\smallskip

The following consequence of the Schur--Zassenhaus theorem is well-known to group-theorists but for the sake of
completeness we nevertheless include a proof.

\begin{lem}\label{lem1}
Let $X$ be a finite group and let $N\trianglelefteq X$ be a normal Hall
$\pi'$-subgroup. Then Hall $\pi$-subgroups of $X$ exist, any two Hall
$\pi$-subgroups of $X$ are conjugate, and every $\pi$-subgroup of $X$ is
contained in a Hall $\pi$-subgroup of $X$.
\end{lem}
\begin{proof}
Most of the statement is the classical Schur--Zassenhaus theorem. It remains
to prove the containment statement. Let $P\leq X$ be a $\pi$-subgroup and
put $L=NP$. Then $N\trianglelefteq L$, the subgroup $N$ is a normal Hall
$\pi'$-subgroup of $L$, and $P$ is a complement of $N$ in $L$. Let $Q$ be
a Hall $\pi$-subgroup of $X$. Since $X=NQ$ and $N\leq L$, Dedekind's
modular law gives $L=L\cap X=L\cap NQ=N(L\cap Q)$. Moreover,
$N\cap(L\cap Q)\leq N\cap Q=1$, so $L\cap Q$ is also a complement of
$N$ in $L$. By the conjugacy part of Schur--Zassenhaus applied inside
$L$, there exists $\ell\in L$ such that $P=(L\cap Q)^\ell$. Therefore
$P\leq Q^\ell$, and $Q^\ell$ is a Hall $\pi$-subgroup of $X$.
\end{proof}

\begin{theo}\label{schurzassenhaus}
Let $B=(B,+,\circ)$ be a finite skew brace, and let $I$ be an ideal of
$B$ such that $\gcd(|I|,|B/I|)=1$. Put $\pi=\pi(|B/I|)$. If $S$ is a
left ideal of $B$ whose order is a $\pi$-number, then there exists a
sub-skew brace $H\leq B$ such that $S\leq H$, $B=I+H=I\circ H$ and
$I\cap H=\{0\}$.
\end{theo}

\begin{proof}
Put $G=(B,+)$. Since $S$ is a left ideal, $S$ is a subgroup of $G$ and
$\lambda_b(S)=S$ for every $b\in B$. In particular, $S$ is a
sub-skew brace of $B$, because $s\circ t=s+\lambda_s(t)\in S$ for all
$s,t\in S$.

We first choose a multiplicative complement of $I$ containing $S$. Since
$I$ is an ideal, $I\trianglelefteq (B,\circ)$; moreover $I$ is a normal
Hall $\pi'$-subgroup of $(B,\circ)$. By Lemma~\ref{lem1}, applied in
the group $(B,\circ)$, there exists a Hall $\pi$-subgroup
$Q\leq (B,\circ)$ such that $S\leq Q$. Thus $Q$ is a complement of $I$
in $(B,\circ)$.

We now choose an additive complement of $I$ containing $S$ and stable
under $\lambda(Q)$. Consider the semidirect product
$G\rtimes_\lambda Q$, where $Q$ acts on $G$ by the lambda maps. Since
$I$ is an ideal, $(I,+)$ is normal in $G$ and $\lambda_q(I)=I$ for every
$q\in Q$. Hence $I$ is a normal Hall $\pi'$-subgroup of
$G\rtimes_\lambda Q$. The subgroup $S\rtimes Q$ is a $\pi$-subgroup of
$G\rtimes_\lambda Q$, because $S$ is $\lambda(Q)$-invariant and both
$S$ and $Q$ are $\pi$-groups. Applying Lemma~\ref{lem1} to
$G\rtimes_\lambda Q$, there exists a Hall $\pi$-subgroup
$R$ of $G\rtimes_\lambda Q$ such that $S\rtimes Q\leq R$. 
Since \(R\) contains \(\{0\}\rtimes Q\), the projection
\(R\to Q\) is onto. Hence
\[
        |R\cap G|=|R|/|Q|=|B/I|.
\]
Thus \(P=R\cap G\) is a Hall \(\pi\)-subgroup of \(G=(B,+)\). Moreover $S\leq P$, and $P$ is
$\lambda(Q)$-invariant, because $Q\leq R$ normalizes $R\cap G$. Thus
$P$ is a complement of $I$ in $(B,+)$, with $S\leq P$ and
$\lambda_q(P)=P$ for every $q\in Q$.

We now pass to the holomorph of $G$. Let
$A=G\rtimes \lambda(Q)$. Since $P$ is $\lambda(Q)$-invariant,
$P\rtimes\lambda(Q)$ is a subgroup of $A$. Moreover it is a Hall
$\pi$-subgroup of $A$, because $P$ is a Hall $\pi$-subgroup of $G$ and
$\lambda(Q)$ is a $\pi$-group. The map
$a\mapsto (a,\lambda_a)$ embeds $(B,\circ)$ in $\operatorname{Hol}(G)$;
write $\ell_\circ(a)=(a,\lambda_a)$. Then $\ell_\circ(Q)$ is a
$\pi$-subgroup of $A$.

Since $I$ is $\lambda(Q)$-invariant, $I$ is a normal Hall
$\pi'$-subgroup of $A=G\rtimes\lambda(Q)$. Therefore, by
Lemma~\ref{lem1}, the $\pi$-subgroup $\ell_\circ(Q)$ is contained in a
Hall $\pi$-subgroup of $A$. Since all Hall $\pi$-subgroups of $A$ are
conjugate, there exist $g\in G$ and $h\in Q$ such that
\[
        \ell_\circ(Q)\leq
        (g,\lambda_h)(P\rtimes\lambda(Q))(g,\lambda_h)^{-1}.
\]
Equivalently,
$\ell_\circ(Q)(g,\lambda_h)\subseteq
(g,\lambda_h)(P\rtimes\lambda(Q))$. Evaluating both sides at the
identity element of $G$, we obtain $Q\circ g\subseteq g+\lambda_h(P)$.
Both sets have order $|Q|=|P|=|B/I|$, and hence
$Q\circ g=g+\lambda_h(P)$.

Set $H=g^{-1,\circ}\circ Q\circ g$. Then $H$ is a subgroup of
$(B,\circ)$ and $|H|=|Q|=|B/I|$. We claim that $H$ is also a subgroup
of $(B,+)$. Indeed, from $Q\circ g=g+\lambda_h(P)$ we get
$-g+Q\circ g=\lambda_h(P)$. On the other hand, for every $q\in Q$,
\[
        -g+q\circ g
        =
        \lambda_g(g^{-1,\circ}\circ q\circ g).
\]
Thus $\lambda_g(H)=\lambda_h(P)$, and therefore
$H=\lambda_g^{-1}\lambda_h(P)=\lambda_{g^{-1,\circ}\circ h}(P)$.
Since $P$ is a subgroup of $(B,+)$ and $\lambda_{g^{-1,\circ}\circ h}$
is an automorphism of $(B,+)$, it follows that $H$ is a subgroup of~$(B,+)$.

Therefore $H$ is a subgroup of both $(B,+)$ and $(B,\circ)$, and hence
$H$ is a sub-skew brace of $B$. Moreover $S\leq P$ and $S$ is a left
ideal, so
$S=\lambda_{g^{-1,\circ}\circ h}(S)\leq
\lambda_{g^{-1,\circ}\circ h}(P)=H$. Thus $S\leq H$.

Finally, $|H|=|B/I|$. Since $|I|$ and $|H|$ are coprime, we have
$I\cap H=\{0\}$. Since $I$ is an ideal, it is normal in both
$(B,+)$ and $(B,\circ)$. Hence $I+H$ and $I\circ H$ are subgroups, and
$|I+H|=|I\circ H|=|I||H|=|B|$. Therefore
$B=I+H=I\circ H$, as required.
\end{proof}

\begin{rem}
{\rm The existence statement in Theorem~\ref{schurzassenhaus} can also
be proved by adapting the trifactorised-group approach to finite skew
braces, as in \cite{BallesterPerezPerez26}. This alternative proof,
however, only gives the existence of a skew brace complement of a Hall
ideal. It does not seem to provide the additional containment property
for prescribed left ideals.}
\end{rem}

\medskip

The following example shows that the full analog of Lemma \ref{lem1} does not hold for skew braces.

\begin{ex}\label{ex1}
{\rm Let $N=\mathbb F_3$, written additively, and let
$Q=\langle r,t\mid r^2=t^2=1,\ rt=tr\rangle\simeq C_2\times C_2$.
Define two actions $\alpha,\tau\colon Q\to \operatorname{Aut}(N)$ by
$\alpha_r=-1$, $\alpha_t=1$, $\tau_r=1$, and $\tau_t=-1$. Put
$\beta_q=\alpha_q\tau_q$ for every $q\in Q$.

Let $B=N\times Q$. Define two operations on $B$ by
$(a,q)+(b,u)=(a+\alpha_q(b),qu)$ and
$(a,q)\circ(b,u)=(a+\beta_q(b),qu)$. Then $(B,+)$ and $(B,\circ)$ are
the semidirect products $N\rtimes_\alpha Q$ and $N\rtimes_\beta Q$,
respectively. Moreover, for $x=(a,q)$ and $y=(b,u)$ one has
\[
        \lambda_x(y)=-_+x+x\circ y=(\tau_q(b),u).
\]
Since $\tau_q$ commutes with every $\alpha_u$, the map
$y\mapsto \lambda_x(y)$ is an automorphism of $(B,+)$; and since
$\tau$ is a homomorphism, $x\mapsto \lambda_x$ is a homomorphism from
$(B,\circ)$ to $\operatorname{Aut}(B,+)$. Also
$x\circ y=x+\lambda_x(y)$. Hence $B=(B,+,\circ)$ is a skew brace.

Let $I=N\times\{1\}$. Then $I$ is an ideal of $B$: it is normal in both
$(B,+)$ and $(B,\circ)$, and it is invariant under all $\lambda_x$.
Moreover $|I|=3$ and $|B/I|=4$, so $I$ is a Hall ideal. Put
$\pi=\{2\}$.

Now let $S=\{(0,1),(1,r)\}$. Since $\alpha_r=\beta_r=-1$, we have
$(1,r)+(1,r)=(0,1)$ and $(1,r)\circ(1,r)=(0,1)$. Hence $S$ is a
sub-skew brace of order $2$. Also $S\cap I=\{(0,1)\}$.

We claim that $S$ is not contained in any skew brace complement of $I$.
Let $H$ be a skew brace complement of $I$ in $B$. Since $H\cap I=1$ and
$|H|=|B/I|=4$, the projection $B\to B/I\simeq Q$ restricts to a
bijection $H\to Q$. Thus there is a function $F\colon Q\to N$, with
$F(1)=0$, such that $H=\{(F(q),q):q\in Q\}$.

Since $H$ is closed under $+$, we have
$F(qu)=F(q)+\alpha_q(F(u))$ for all $q,u\in Q$. Since $H$ is closed
under $\circ$, we also have $F(qu)=F(q)+\beta_q(F(u))$ for all
$q,u\in Q$. Therefore $\alpha_q(F(u))=\beta_q(F(u))$ for all $q,u$.
Since $\beta_q=\alpha_q\tau_q$, this gives $F(u)=\tau_q(F(u))$ for all
$q,u\in Q$.

Taking $q=t$, we get $F(u)=\tau_t(F(u))=-F(u)$ for every $u\in Q$.
Since $N=\mathbb F_3$, this forces $F(u)=0$ for every $u\in Q$.
Therefore the only skew brace complement of $I$ is
$H_0=\{0\}\times Q$.

But $(1,r)\in S$ and $(1,r)\notin H_0$. Hence $S$ is not contained in
any skew brace complement of $I$.}
\end{ex}

Example~\ref{ex1} shows that, unlike in group theory, Sylow sub-skew
braces do not contain all sub-skew braces of prime-power order. Hence
the recent Sylow existence theorems for finite skew braces
\cite{BallesterPerezPerez26,Truman26} cannot be strengthened in this
direction without further assumptions. We prove below that the expected
containment property is recovered for left ideals: every left ideal of
prime-power order is contained in a Sylow sub-skew brace. This refines
the Sylow existence results in this special case and is in line with the
containment phenomena appearing in
\cite{CarantiDelCorsoDiMatteoFerraraTrombetti25,ErcanGulGulogluKizmaz26}.
For completeness, we give the proof.

\begin{theo}\label{sylowtheorem}
Let $B=(B,+,\circ)$ be a finite skew brace, let $p$ be a prime, and let
$S$ be a left ideal of $B$ whose order is a power of $p$. Then there
exists a Sylow $p$-sub-skew brace $H$ of $B$ such that $S\leq H$.
\end{theo}

\begin{proof}
We use the holomorph of the additive group. Put $G=(B,+)$, and identify
$\operatorname{Hol}(G)$ with $G\rtimes\operatorname{Aut}(G)$. Thus
$(x,\alpha)(y,\beta)=(x+\alpha(y),\alpha\beta)$, and
$(x,\alpha)$ acts on $G$ by $(x,\alpha)[y]=x+\alpha(y)$.

For every $a\in B$, let $\lambda_a\in\operatorname{Aut}(G)$ be the
lambda map, $\lambda_a(b)=-a+a\circ b$. Then
$a\circ b=a+\lambda_a(b)$, and the map
$a\mapsto (a,\lambda_a)$ embeds $(B,\circ)$ into $\operatorname{Hol}(G)$.
We denote this embedding by $\ell_\circ$, so that
$\ell_\circ(a)=(a,\lambda_a)$.

Since $S$ is a left ideal, $S$ is a subgroup of $(B,+)$ and
$\lambda_b(S)=S$ for every $b\in B$. In particular, $S$ is also a
subgroup of $(B,\circ)$, because if $s,t\in S$, then
$s\circ t=s+\lambda_s(t)\in S$.

Choose a Sylow $p$-subgroup $Q$ of $(B,\circ)$ such that $S\leq Q$.
We now choose a suitable Sylow $p$-subgroup of $(B,+)$. Let $\Omega$ be
the set of all Sylow $p$-subgroups of $(B,+)$ which contain $S$.

First, $\Omega$ is non-empty by the ordinary Sylow theorem. Moreover
$|\Omega|\equiv 1\pmod p$. Indeed, let $\operatorname{Syl}_p(G)$ be the
set of all Sylow $p$-subgroups of $G=(B,+)$, and let $S$ act on
$\operatorname{Syl}_p(G)$ by conjugation. The fixed points of this
action are precisely the Sylow $p$-subgroups of $G$ which contain $S$:
if a Sylow $p$-subgroup $P$ contains $S$, then $S$ normalizes $P$; and
if $S$ normalizes $P$, then $SP$ is a $p$-subgroup of $G$, hence
$SP=P$ and so $S\leq P$. Since
$|\operatorname{Syl}_p(G)|\equiv 1\pmod p$ and all non-trivial orbits
of the $p$-group $S$ have length divisible by $p$, it follows that
$|\Omega|\equiv 1\pmod p$.

The group $Q$ acts on $\Omega$ via the lambda action. Indeed, if
$P\in\Omega$ and $q\in Q$, then $\lambda_q(P)$ is again a Sylow
$p$-subgroup of $(B,+)$; moreover, since $S$ is a left ideal, we have
$\lambda_q(S)=S$, and hence $S\leq\lambda_q(P)$. Thus
$\lambda_q(P)\in\Omega$.

Since $Q$ is a $p$-group and $|\Omega|\equiv 1\pmod p$, the action of
$Q$ on $\Omega$ has a fixed point. Hence there exists a Sylow
$p$-subgroup $P$ of $(B,+)$ such that $S\leq P$ and
$\lambda_q(P)=P$ for every $q\in Q$.

Now consider the subgroup $P\rtimes\lambda(Q)$ of
$G\rtimes\lambda(B)$. It is a Sylow $p$-subgroup of
$G\rtimes\lambda(B)$: indeed, $P$ is Sylow in $G$, and $\lambda(Q)$ is
Sylow in $\lambda(B)$, since $\lambda\colon (B,\circ)\to\lambda(B)$ is
an epimorphism and $Q$ is Sylow in $(B,\circ)$.


The subgroup $\ell_\circ(Q)$ is a $p$-subgroup of $G\rtimes\lambda(B)$.
By the ordinary Sylow theorem, $\ell_\circ(Q)$ is contained in a
conjugate of $P\rtimes\lambda(Q)$. Thus there exist $g,h\in B$ such
that
$\ell_\circ(Q)\leq (g,\lambda_h)(P\rtimes\lambda(Q))(g,\lambda_h)^{-1}$.
Equivalently,
$\ell_\circ(Q)(g,\lambda_h)\leq (g,\lambda_h)(P\rtimes\lambda(Q))$.

Evaluating both sides at the identity element of $(B,+)$ gives
$Q\circ g\subseteq g+\lambda_h(P)$. Since both sets have size $|B|_p$,
we have $Q\circ g=g+\lambda_h(P)$.

Set $H=g^{-1,\circ}\circ Q\circ g$. Then $H$ is a Sylow $p$-subgroup of
$(B,\circ)$. We now show that it is also a Sylow $p$-subgroup of
$(B,+)$. From $Q\circ g=g+\lambda_h(P)$ we get
$-g+Q\circ g=\lambda_h(P)$. On the other hand,
$-g+q\circ g=\lambda_g(g^{-1,\circ}\circ q\circ g)$ for every
$q\in Q$. Hence $\lambda_g(H)=\lambda_h(P)$, and therefore
$H=\lambda_g^{-1}\lambda_h(P)=\lambda_{g^{-1,\circ}\circ h}(P)$.
Thus $H$ is a Sylow $p$-subgroup of~$(B,+)$.

So $H$ is a subgroup of both $(B,+)$ and $(B,\circ)$, and therefore it
is a sub-skew brace of $B$. Finally, since $S\leq P$ and $S$ is a left
ideal, we have
$\lambda_{g^{-1,\circ}\circ h}(S)=S$. Hence
$S\leq \lambda_{g^{-1,\circ}\circ h}(P)=H$. Therefore $H$ is a Sylow
$p$-sub-skew brace of $B$ containing $S$.
\end{proof}

\begin{theo}\label{sylowtheorem2}
Let $B=(B,+,\circ)$ be a finite skew brace such that both groups
$(B,+)$ and $(B,\circ)$ are soluble. Let $\pi$ be a set of primes, and
let $S$ be a left ideal of $B$ whose order is a $\pi$-number. Then there
exists a Hall $\pi$-sub-skew brace $H$ of $B$ such that $S\leq H$.
\end{theo}

\begin{proof}
Put $G=(B,+)$, and identify $\operatorname{Hol}(G)$ with
$G\rtimes \operatorname{Aut}(G)$. Thus
$(x,\alpha)(y,\beta)=(x+\alpha(y),\alpha\beta)$, and $(x,\alpha)$ acts
on $G$ by $(x,\alpha)[y]=x+\alpha(y)$.

For every $a\in B$, let $\lambda_a\in\operatorname{Aut}(G)$ be the
lambda map, $\lambda_a(b)=-a+a\circ b$. Then
$a\circ b=a+\lambda_a(b)$, and the map
$a\mapsto (a,\lambda_a)$ embeds $(B,\circ)$ into
$\operatorname{Hol}(G)$. We denote this embedding by $\ell_\circ$.

Since $S$ is a left ideal, $S$ is a subgroup of $(B,+)$ and
$\lambda_b(S)=S$ for every $b\in B$. In particular, $S$ is a subgroup
of $(B,\circ)$, because if $s,t\in S$, then
$s\circ t=s+\lambda_s(t)\in S$.

Since $(B,\circ)$ is soluble, Hall's theorem for finite soluble groups
gives a Hall $\pi$-subgroup $Q$ of $(B,\circ)$ such that $S\leq Q$.

We next choose a Hall $\pi$-subgroup of $(B,+)$ which contains $S$ and
is invariant under $\lambda(Q)$. Consider the subgroup
$G\rtimes\lambda(Q)$ of $G\rtimes\lambda(B)$. This group is soluble.
Moreover $S\rtimes\lambda(Q)$ is a $\pi$-subgroup of
$G\rtimes\lambda(Q)$, since $S$ is $\lambda(Q)$-invariant. Hence, by
Hall's theorem applied to $G\rtimes\lambda(Q)$, there exists a Hall
$\pi$-subgroup $M$ of $G\rtimes\lambda(Q)$ such that
$S\rtimes\lambda(Q)\leq M$.

Put $P=M\cap G$. Since $G\trianglelefteq G\rtimes\lambda(Q)$, the
subgroup $P$ is a Hall $\pi$-subgroup of $G$. Moreover $S\leq P$.
Since $\lambda(Q)\leq M$ and $P=M\cap G\trianglelefteq M$, we have
$\lambda_q(P)=P$ for every $q\in Q$. Thus
$P\rtimes\lambda(Q)$ is a subgroup of $G\rtimes\lambda(B)$.

We claim that $P\rtimes\lambda(Q)$ is a Hall $\pi$-subgroup of
$G\rtimes\lambda(B)$. Indeed, $P$ is a Hall $\pi$-subgroup of $G$, and
$\lambda(Q)$ is a Hall $\pi$-subgroup of $\lambda(B)$, because
$\lambda\colon (B,\circ)\to \lambda(B)$ is an epimorphism and $Q$ is a
Hall $\pi$-subgroup of $(B,\circ)$. Hence
\[
        |P\rtimes\lambda(Q)|
        =
        |P|\,|\lambda(Q)|
        =
        |G|_\pi\,|\lambda(B)|_\pi
        =
        |G\rtimes\lambda(B)|_\pi.
\]
Thus $P\rtimes\lambda(Q)$ is a Hall $\pi$-subgroup of
$G\rtimes\lambda(B)$.

The subgroup $\ell_\circ(Q)$ is a $\pi$-subgroup of
$G\rtimes\lambda(B)$. Since $G\rtimes\lambda(B)$ is soluble, Hall's
theorem implies that $\ell_\circ(Q)$ is contained in a conjugate of
$P\rtimes\lambda(Q)$. Hence there exist $g\in G$ and $h\in B$ such that
\[
        \ell_\circ(Q)
        \leq
        (g,\lambda_h)(P\rtimes\lambda(Q))(g,\lambda_h)^{-1}.
\]
Equivalently,
\[
        \ell_\circ(Q)(g,\lambda_h)
        \subseteq
        (g,\lambda_h)(P\rtimes\lambda(Q)).
\]

Applying both sides to the identity element $0$ of $(B,+)$, we get
$Q\circ g\subseteq g+\lambda_h(P)$. Indeed, for $q\in Q$,
$\ell_\circ(q)(g,\lambda_h)(0)=q\circ g$, while for $p\in P$ and
$q\in Q$, $(g,\lambda_h)(p,\lambda_q)(0)=g+\lambda_h(p)$. Since both
sets have order $|B|_\pi$, it follows that
\[
        Q\circ g=g+\lambda_h(P).
\]

Set $H=g^{-1,\circ}\circ Q\circ g$. Then $H$ is a Hall $\pi$-subgroup
of $(B,\circ)$. We show that $H$ is also a Hall $\pi$-subgroup of
$(B,+)$. From $Q\circ g=g+\lambda_h(P)$ we get
$-g+Q\circ g=\lambda_h(P)$. On the other hand, for every $q\in Q$,
\[
        -g+q\circ g
        =
        \lambda_g(g^{-1,\circ}\circ q\circ g).
\]
Therefore $\lambda_g(H)=\lambda_h(P)$, and hence
$H=\lambda_g^{-1}\lambda_h(P)=\lambda_{g^{-1,\circ}\circ h}(P)$. Since
$\lambda_{g^{-1,\circ}\circ h}$ is an automorphism of $(B,+)$ and $P$
is a Hall $\pi$-subgroup of $(B,+)$, it follows that $H$ is a Hall
$\pi$-subgroup of $(B,+)$.

Thus $H$ is a subgroup of both $(B,+)$ and $(B,\circ)$, and so $H$ is a
sub-skew brace of $B$. Finally, since $S\leq P$ and $S$ is a left ideal,
we have $\lambda_{g^{-1,\circ}\circ h}(S)=S$. Hence
$S\leq \lambda_{g^{-1,\circ}\circ h}(P)=H$. Therefore $H$ is a Hall
$\pi$-sub-skew brace of $B$ containing $S$.
\end{proof}

\medskip

Our next example shows that the Sylow sub-skew braces may even not be isomorphic even in the case of a soluble skew brace.

\begin{ex}\label{ex2}
{\rm Let $G=S_4$ and put $a=(1\,2)(3\,4)$. Define $\psi\colon G\to G$ by
\[
        \psi(g)=
        \begin{cases}
        1, & g\in A_4,\\
        a, & g\notin A_4.
        \end{cases}
\]
Then $\psi$ is an endomorphism of $G$, with abelian image
$\psi(G)=\langle a\rangle$. Hence we may define a skew brace
$B=(G,+,\circ)$ by taking $x+y=xy$ and
$x\circ y=x\psi(x^{-1})y\psi(x)$ for all $x,y\in G$. In this skew brace
$\lambda_x(y)=\psi(x^{-1})y\psi(x)$; hence $\lambda_x$ is the identity
if $x$ is even, and is conjugation by $a$ if $x$ is odd.

The skew brace $B$ is soluble. Indeed, $(B,+)\simeq S_4$ is soluble, and
$(B,\circ)$ embeds into $S_4\rtimes \lambda(B)$, where $\lambda(B)$ has
order $2$. Hence $(B,\circ)$ is soluble as well.

Let
\[
        H_1=\langle (1\,2),\ (3\,4),\ (1\,3)(2\,4)\rangle,
        \qquad
        H_2=\langle (1\,2\,3\,4),\ (1\,2)(3\,4)\rangle.
\]
Then $H_1=C_{S_4}(a)$, so $|H_1|=8$. Also $H_2$ has order $8$: if
$r=(1\,2\,3\,4)$, then $H_2=\langle r,a\rangle$, with $r^4=a^2=1$ and
$ara=r^{-1}$. Thus both $H_1$ and $H_2$ are Sylow $2$-subgroups of
$(B,+)$.

Both $H_1$ and $H_2$ contain $a$. Therefore conjugation by $a$ stabilises
both of them. Since every $\lambda_h$, with $h\in H_i$, is either the
identity or conjugation by $a$, it follows that $\lambda_h(H_i)=H_i$ for
all $h\in H_i$. Hence $H_1$ and $H_2$ are sub-skew braces of $B$, and so
they are Sylow $2$-sub-skew braces.

We claim that $H_1$ and $H_2$ are not isomorphic as skew braces. Since
$H_1=C_{S_4}(a)$, the element $a$ centralises $H_1$. Hence, for every
$h\in H_1$, the restriction of $\lambda_h$ to $H_1$ is the identity.
Indeed, if $h$ is even this is clear, while if $h$ is odd then
$\lambda_h$ is conjugation by $a$, which is trivial on $H_1$. Therefore
$h\circ y=h+\lambda_h(y)=h+y$ for all $h,y\in H_1$, and $H_1$ is a
trivial skew brace.

On the other hand, $H_2$ is not trivial. Indeed, with $r=(1\,2\,3\,4)$,
we have $r\in H_2$ and $r$ is odd. Hence $\lambda_r$ is conjugation by
$a$, and so $\lambda_r(r)=ara=r^{-1}\neq r$. Therefore
$r\circ r=r+\lambda_r(r)=rr^{-1}=1$, whereas $r+r=r^2\neq 1$. Thus the
two operations do not coincide on $H_2$.

Since being trivial is preserved under skew brace isomorphisms, $H_1$
and $H_2$ are not isomorphic as skew braces. Consequently, even for a
finite skew brace whose additive and multiplicative groups are soluble,
Sylow sub-skew braces need not be isomorphic.}
\end{ex}

Finally, we prove the Sylow's Third Theorem for finite skew braces.

\begin{theo}\label{thirdsylow}
Let $B=(B,+,\circ)$ be a finite skew brace, let $p$ be a prime, and let
$n_p(B)$ denote the number of Sylow $p$-sub-skew braces of $B$. Then
$n_p(B)\equiv 1 \pmod p$.
\end{theo}
\begin{proof}
Put $G=(B,+)$, and write $|B|_p=p^a$ and $m=|B|/p^a$. Let
$\mathcal P$ be the set of Sylow $p$-subgroups of $(B,+)$, and let
$\mathcal Q$ be the set of Sylow $p$-subgroups of $(B,\circ)$. Consider
the set
\[
        \mathcal C=\{(P,Q)\in\mathcal P\times\mathcal Q:
        \lambda_q(P)=P \text{ for every } q\in Q\}.
\]
We first prove that $|\mathcal C|\equiv 1\pmod p$. Fix
$Q\in\mathcal Q$. Since $Q$ is a $p$-group, it acts on $\mathcal P$ by
$P\mapsto\lambda_q(P)$. The fixed points of this action are precisely
the Sylow $p$-subgroups $P$ of $(B,+)$ such that $\lambda_q(P)=P$ for
every $q\in Q$. Since a $p$-group acting on a finite set has number of
fixed points congruent modulo $p$ to the cardinality of the whole set,
and since $|\mathcal P|\equiv 1\pmod p$ by the ordinary Sylow theorem,
the number of such $P$ is congruent to $1$ modulo $p$. Summing over all
$Q\in\mathcal Q$, and using $|\mathcal Q|\equiv 1\pmod p$, we get
$|\mathcal C|\equiv 1\pmod p$.

Now fix $(P,Q)\in\mathcal C$. The group $Q$ acts on the set of left
cosets $G/P$ by
\[
        q\cdot(g+P)=q\circ g+P .
\]
This is well-defined. Indeed, if $g'=g+x$ with $x\in P$, then
\[
        q\circ g'
        =
        q+\lambda_q(g+x)
        =
        q+\lambda_q(g)+\lambda_q(x)
        =
        q\circ g+\lambda_q(x),
\]
and $\lambda_q(x)\in P$. Thus $q\circ g'+P=q\circ g+P$. Since
$|G/P|=m$ is prime to $p$, the number of fixed points of the action of
$Q$ on $G/P$ is congruent to $m$ modulo $p$.

Let $\mathcal T$ be the set of triples $(P,Q,g+P)$ such that
$(P,Q)\in\mathcal C$ and $g+P$ is fixed by $Q$ under the action above.
By the previous paragraph,
\[
        |\mathcal T|\equiv m|\mathcal C|\equiv m \pmod p.
\]

We now count $\mathcal T$ in a second way. Let $(P,Q,g+P)\in\mathcal T$.
Since $g+P$ is fixed by $Q$, we have $q\circ g\in g+P$ for every
$q\in Q$. Hence $Q\circ g\subseteq g+P$. Both sets have cardinality
$p^a$, so $Q\circ g=g+P$.

Define
\[
        H=g^{-1}_\circ\circ Q\circ g.
\]
This is well-defined from the triple, not from the chosen representative
$g$. Indeed, if $g'\in g+P$, then, since $Q\circ g=g+P$, there exists
$q_0\in Q$ such that $g'=q_0\circ g$. Therefore
\[
        g_\circ'^{-1}\circ Q\circ g'
        =
        g_\circ^{-1}\circ q_{0,\circ}^{-1}\circ Q\circ q_0\circ g
        =
        g_\circ^{-1}\circ Q\circ g.
\]
Thus $H$ depends only on $(P,Q,g+P)$.

Clearly $H$ is a Sylow $p$-subgroup of $(B,\circ)$. Moreover, for every
$q\in Q$ one has
\[
        \lambda_g(g^{-1}_\circ\circ q\circ g)
        =
        -g+q\circ g.
\]
Thus
\[
        \lambda_g(H)=-g+Q\circ g=-g+(g+P)=P.
\]
Since $\lambda_g$ is an automorphism of $(B,+)$, it follows that $H$ is
a Sylow $p$-subgroup of $(B,+)$. Hence $H$ is a subgroup of both
$(B,+)$ and $(B,\circ)$. Therefore $H$ is a sub-skew brace of $B$, since
for $h,k\in H$ we have $\lambda_h(k)=-h+h\circ k\in H$. Thus $H$ is a
Sylow $p$-sub-skew brace of $B$.

Conversely, let $H$ be a Sylow $p$-sub-skew brace of $B$ and let
$g\in B$. Put
\[
        P_g=\lambda_g(H),\qquad
        Q_g=g\circ H\circ g^{-1}_\circ,\qquad
        \omega_g=g+P_g.
\]
Then $P_g\in\mathcal P$ and $Q_g\in\mathcal Q$. If
$q=g\circ h\circ g^{-1}_\circ$ with $h\in H$, then
\[
        \lambda_q(P_g)
        =
        \lambda_q\lambda_g(H)
        =
        \lambda_{q\circ g}(H)
        =
        \lambda_{g\circ h}(H)
        =
        \lambda_g\lambda_h(H)
        =
        \lambda_g(H)
        =
        P_g,
\]
because $H$ is a sub-skew brace and hence $\lambda_h(H)=H$. Therefore
$(P_g,Q_g)\in\mathcal C$. Moreover
$q\circ g=g\circ h\in g\circ H=g+\lambda_g(H)=g+P_g$, so $\omega_g$ is
fixed by $Q_g$. Hence $(P_g,Q_g,\omega_g)\in\mathcal T$.

We claim that, for a fixed Sylow $p$-sub-skew brace $H$, the triples
obtained in this way are parametrised by the left cosets of $H$ in
$(B,\circ)$. If $h\in H$, then
\[
        P_{g\circ h}
        =
        \lambda_{g\circ h}(H)
        =
        \lambda_g\lambda_h(H)
        =
        P_g,
\]
and
\[
        Q_{g\circ h}
        =
        g\circ h\circ H\circ h^{-1}_\circ\circ g^{-1}_\circ
        =
        g\circ H\circ g^{-1}_\circ
        =
        Q_g.
\]
Moreover $(g\circ h)+P_g=g+\lambda_g(h)+P_g=g+P_g$, because
$\lambda_g(h)\in P_g$. Hence $g$ and $g\circ h$ give the same triple.

Conversely, suppose that $g_1$ and $g_2$ give the same triple. Then
$P_{g_1}=P_{g_2}$ and $g_1+P_{g_1}=g_2+P_{g_2}$. Hence
$g_2\in g_1+P_{g_1}$. Since $P_{g_1}=\lambda_{g_1}(H)$, we have
\[
        g_1+P_{g_1}
        =
        g_1+\lambda_{g_1}(H)
        =
        g_1\circ H.
\]
Thus $g_2\in g_1\circ H$, so $g_1$ and $g_2$ lie in the same left coset
of $H$ in $(B,\circ)$.

Finally, every triple in $\mathcal T$ arises from the Sylow sub-skew
brace $H=g^{-1}_\circ\circ Q\circ g$ constructed above, and the
preceding paragraph shows that, for each fixed $H$, exactly
$|(B,\circ):H|=m$ triples arise. Therefore
\[
        |\mathcal T|=m\,n_p(B).
\]
Since also $|\mathcal T|\equiv m\pmod p$ and $p\nmid m$, we conclude
that $n_p(B)\equiv 1\pmod p$.
\end{proof}

\medskip

Usually, in group theory, Sylow's Third Theorem is used in combination with the fact that the number of Sylow subgroups divides the order of the finite group. In the context of skew braces this is not possible.

\begin{ex}\label{ex:number-not-dividing}
{\rm  Let \(H=\operatorname{GL}(3,2)\), let \(C_2=\langle c\rangle\), and put
\(G=H\times C_2\). Since
\(|H|=(2^3-1)(2^3-2)(2^3-2^2)=7\cdot 6\cdot 4=168\), we have
\(|G|=336\). Fix an involution \(a\in H\), and put \(a_0=(a,1)\in G\).
Define \(\psi\colon G\to G\) by \(\psi(h,c^\varepsilon)=a_0^\varepsilon\),
for \(h\in H\) and \(\varepsilon\in\{0,1\}\). Then \(\psi\) is an
endomorphism of \(G\), and \(\psi(G)=\langle a_0\rangle\) is abelian.
Therefore the standard construction gives a skew brace structure on
\(B=G\), namely
\[
        x+y=xy,\qquad
        x\circ y=x\psi(x^{-1})y\psi(x)
        \qquad (x,y\in G).
\]
Indeed, the associated maps are given by
\(\lambda_x(y)=\psi(x^{-1})y\psi(x)\), and the fact that \(\psi(G)\) is
abelian gives \(\lambda_{x\circ y}=\lambda_x\lambda_y\).

We claim that \(B\) has exactly \(5\) Sylow \(2\)-sub-skew braces. Let
\(P\) be a Sylow \(2\)-subgroup of the additive group \((B,+)=G\). Since
the factor \(C_2\) is a normal \(2\)-subgroup of \(G\), every Sylow
\(2\)-subgroup of \(G\) contains it. Hence \(P=S\times C_2\), where
\(S\in\operatorname{Syl}_2(H)\).

Now \(P\) is already a subgroup of \((B,+)\), so it is a sub-skew brace
if and only if it is closed under \(\circ\). For \(x,y\in P\), the
condition \(x\circ y\in P\) is equivalent to
\(\psi(x^{-1})y\psi(x)\in P\). Thus \(P\) is a sub-skew brace if and
only if \(\psi(P)\leq N_G(P)\). Since \(P=S\times C_2\), we have
\(\psi(P)=\langle a_0\rangle\). Therefore
\[
        P \text{ is a Sylow }2\text{-sub-skew brace}
        \Longleftrightarrow
        a_0\in N_G(P)
        \Longleftrightarrow
        aSa^{-1}=S
        \Longleftrightarrow
        a\in S .
\]
The last equivalence follows because \(a\) has order \(2\): if \(a\)
normalizes \(S\), then \(\langle S,a\rangle\) is a \(2\)-subgroup of
\(H\), hence \(\langle S,a\rangle=S\) by maximality of \(S\).

It remains to count the Sylow \(2\)-subgroups of \(H\) which contain the
fixed involution \(a\). Since \(|H|=168=2^3\cdot 3\cdot 7\), a Sylow
\(2\)-subgroup of \(H\) has order \(8\). The upper unitriangular group
\(UT_3(2)\) is such a Sylow \(2\)-subgroup. It is self-normalizing in
\(H\): its normalizer is the upper triangular subgroup, which coincides
with \(UT_3(2)\) because \(\mathbb F_2^\times=\{1\}\). Hence the number
of Sylow \(2\)-subgroups of \(H\) is \(168/8=21\).

Each Sylow \(2\)-subgroup contains exactly \(5\) involutions. Indeed,
\(UT_3(2)\) consists of the matrices
\(I+\alpha E_{12}+\beta E_{13}+\gamma E_{23}\), with
\(\alpha,\beta,\gamma\in\mathbb F_2\). Squaring gives
\((I+\alpha E_{12}+\beta E_{13}+\gamma E_{23})^2
=I+\alpha\gamma E_{13}\). Thus a non-identity element has order \(2\)
precisely when \(\alpha\gamma=0\). Among the \(8\) triples
\((\alpha,\beta,\gamma)\), there are \(6\) with \(\alpha\gamma=0\);
one of them gives the identity, so there are \(5\) involutions.

The group \(H\) has exactly \(21\) involutions. Indeed, in characteristic
\(2\), an involution is of the form \(I+N\), with \(0\neq N\) and
\(N^2=0\). In dimension \(3\), this forces \(\operatorname{rank}(N)=1\),
so \(N=u\otimes f\), where \(0\neq u\in\mathbb F_2^3\),
\(0\neq f\in(\mathbb F_2^3)^*\), and \(f(u)=0\). There are \(7\) choices
for \(u\), and for each fixed \(u\) the annihilator of \(u\) in
\((\mathbb F_2^3)^*\) has dimension \(2\), hence contains \(3\) nonzero
linear forms. Since there is no nontrivial scalar ambiguity over
\(\mathbb F_2\), this gives \(7\cdot 3=21\) involutions.

Now count explicitly the pairs \((S,t)\), where
\(S\in\operatorname{Syl}_2(H)\) and \(t\) is an involution of \(S\).
Counting first by \(S\), there are \(21\) choices for \(S\), and each
such \(S\) contains \(5\) involutions, so the number of pairs is
\(21\cdot 5=105\). Counting first by \(t\), there are \(21\) involutions
in \(H\). Since all involutions of \(H\) are conjugate, each involution
lies in the same number, say \(r\), of Sylow \(2\)-subgroups. Thus the
same number of pairs is \(21r\). Therefore \(21r=105\), so \(r=5\).
Thus the fixed involution \(a\) belongs to exactly \(5\) Sylow
\(2\)-subgroups of \(H\). By the criterion above, \(B\) has exactly
\(5\) Sylow \(2\)-sub-skew braces. Consequently
\[
        n_2(B)=5,\qquad |B|=336,\qquad 5\nmid 336.
\]
Therefore the number of Sylow \(2\)-sub-skew braces of \(B\) does not
divide the order of \(B\).}
\end{ex}

\bigskip\bigskip\bigskip

\renewcommand{\bibsection}{\begin{flushright}\Large
{
REFERENCES}\\
\rule{8cm}{0.4pt}\\[0.8cm]
\end{flushright}}

\bigskip\bigskip


\begin{flushleft}
\rule{8cm}{0.4pt}\\
\end{flushleft}

{
\sloppy
\noindent
Maria Ferrara

\noindent
Dipartimento di Ingegneria

\noindent
Facoltà di Ingegneria e Informatica

\noindent
Università Pegaso

\noindent
e-mail: maria.ferrara1@unipegaso.it
}

\bigskip
\bigskip

{
\sloppy
\noindent
Marco Trombetti

\noindent 
Dipartimento di Matematica e Applicazioni ``Renato Caccioppoli''

\noindent
Università degli Studi di Napoli Federico II

\noindent
Complesso Universitario Monte S. Angelo

\noindent
Via Cintia, Napoli (Italy)

\noindent
e-mail: marco.trombetti@unina.it 

}



\begin{thebibliography}{99} 


\baselineskip 10pt
{

\bibitem{BallesterPerezPerez26}
{\ssc A. Ballester-Bolinches -- P. Pérez-Altarriba -- V. Pérez-Calabuig}:
``On finite trifactorised groups and Sylow and Hall theorems for skew braces'',
ArXiv:2606.24977 (2026).

\bibitem{CarantiDelCorsoDiMatteoFerraraTrombetti25}
{\ssc A. Caranti -- I. Del Corso -- M. Di Matteo -- M. Ferrara -- M. Trombetti}:
``On the Sylow Theorem for Skew Braces'', ArXiv:2506.00940.

\bibitem{ErcanGulGulogluKizmaz26}
{\ssc G. Ercan -- S. Gül -- I.S. Güloğlu -- M.Y. Kızmaz}:
``Sylow theory and the nilpotency class of left nilpotent skew braces'',
ArXiv:2606.25691.

\bibitem{Huppert}{\ssc B. Huppert}: ``Finite Groups I'', {\it Springer}, Cham (2025).

\bibitem{RatheeYadav26}
{\ssc N. Rathee -- M.K. Yadav}:
``Skew brace extensions, second cohomology and complements'',
ArXiv:2601.12371 (2026).

\bibitem{Rob95} {\ssc D.J.S. Robinson}: ``A Course in the Theory of Groups'', {\it Springer}, New York (1996).


\bibitem{Truman26}{\ssc P.J. Truman}: ``Analogues of Sylow's first theorem, Cauchy's theorem, and Hall's theorem for skew braces'', ArXiv:2606.18414.


}
\end{thebibliography}
\end{document}